\documentclass{article}
\usepackage{amssymb,amsfonts,amsmath,amsthm}
\numberwithin{equation}{section}
\usepackage{enumerate}
\usepackage{diagbox}
\usepackage{threeparttable}
\usepackage{mathtools}
\usepackage{bm}
\usepackage{graphicx}
\usepackage{float}
\usepackage{color}
\usepackage[top=1in, bottom=1in, left=1.25in, right=1.25in]{geometry}

\theoremstyle{definition}
\newtheorem{ex}{Exercise}[section]

\newtheorem*{qus}{Question}
\theoremstyle{plain}
\newtheorem{thm}{Theorem}

\newtheorem{prop}[ex]{Proposition}
\theoremstyle{remark}
\newtheorem{rmk}[ex]{Remark}

\usepackage[utf8]{inputenc}

\begin{document}
		\title{Heisenberg uncertainty principle and its analogues in higher dimension: via Wigdersons' method}
			\author{Yiyu Tang }
			\date{Laboratoire d'analyse et de math\'ematiques appliqu\'ees, Universit\'e Gustave Eiffel}
		\maketitle
		\begin{abstract}
		The following question was proposed by Avi Wigderson and Yuval Wigderson in \cite{refww}: Is it possible to use the method in \cite{refww} to prove Heisenberg uncertainty principle in higher dimension $\mathbb{R}^d$, and get the correct dependence of the constant on $d$? We answer this question affirmatively, and also prove some generalizations of Heisenberg uncertainty principle in $\mathbb{R}^d$ via Wigdersons' method.
		\end{abstract}
	
\section{Introduction}
\paragraph{}Recall the definition of Schwartz class:
\begin{equation*}
	\mathcal{S}(\mathbb{R}^d):=\{f\in C^\infty(\mathbb{R}^d;\mathbb{C}):\rho_{\alpha,\beta}(f)\text{ is finite for all }\alpha=(\alpha_i)_{i=1}^d,\beta=(\beta_i)_{i=1}^d\in\mathbb{N}^d\},
\end{equation*}
where $\rho_{\alpha,\beta}(f)=\sup_{x\in\mathbb{R}^d}|x^\alpha\partial^\beta f(x)|$. Here $x^\alpha=x_1^{\alpha_1}x_2^{\alpha_2}\cdots x_d^{\alpha_d}$, and $\partial^\beta=\partial_1^{\beta_1}\partial_{2}^{\beta_2}\cdots\partial_d^{\beta_d}$. 
\par For a Schwartz function $f\in\mathcal{S}(\mathbb{R}^d)$, its Fourier transform $\hat{f}$ is defined by $\hat{f}(\xi):=\int_{\mathbb{R}^d}f(x)\mathrm{e}^{-2\pi\mathrm{i}x\cdot\xi}\mathrm{d}x$, where $x\cdot\xi$ is the inner product in $\mathbb{R}^d$. The Heisenberg uncertainty principle states that for any $f\in\mathcal{S}(\mathbb{R}^d)$
\begin{equation}\label{Heisenberg uncertainty principle}
	\bigg(\int_{\mathbb{R}^d}|x|^2|f(x)|^2\mathrm{d}x\bigg)\cdot	\bigg(\int_{\mathbb{R}^d}|\xi|^2|\hat{f}(\xi)|^2\mathrm{d}\xi\bigg)\geq\frac{d^2}{16\pi^2}||f||^2_{L^2(\mathbb{R}^d)}\cdot||\hat{f}||^2_{L^2(\mathbb{R}^d)},
\end{equation}
and the equality holds if and only if $f$ has the form $f(x)=C\mathrm{e}^{-\pi\delta|x|^2}$, where $\delta>0$.
\par The classical proof of (\ref{Heisenberg uncertainty principle}) is to write
\begin{equation*}
	d||f||^2_{L^2}=\int_{\mathbb{R}^d}f(x)\overline{f(x)}\sum_{j=1}^d\frac{\partial}{\partial x_j}x_j\mathrm{d}x,\text{ where }x=(x_1,x_2,...,x_d).
\end{equation*}
For $x\in\mathbb{R}^d$, we denote $ |x|:=\sqrt{|x_1|^2+\cdots+|x_d|^2}$, and $\langle\cdot,\cdot\rangle_{\mathbb{R}^d}$ is the inner product in $\mathbb{R}^d$. Using integration by parts, Cauchy-Schwarz inequality, and Plancherel identity:
\begin{equation*}
\begin{aligned}
	d||f||^2_2&=2\cdot\mathrm{Re}\bigg(-\int_{\mathbb{R}^d}\langle\nabla f,\overline{xf}\rangle_{\mathbb{R}^d}\mathrm{d}x\bigg)\\
	&\leq2\bigg|-\int_{\mathbb{R}^d}\langle\nabla f,\overline{xf}\rangle_{\mathbb{R}^d}\mathrm{d}x\bigg|\\
	&\leq2\int_{\mathbb{R}^d}|\nabla f|\cdot|\overline{xf}|\mathrm{d}x\\
	&\leq2\bigg(\int_{\mathbb{R}^d}|\nabla f|^2\mathrm{d}x\bigg)^\frac{1}{2}\cdot \bigg(\int_{\mathbb{R}^d}|x|^2|f|^2\mathrm{d}x\bigg)^\frac{1}{2}\\
	&=2\bigg(\int_{\mathbb{R}^d}\sum_{j=1}^d|\widehat{\partial_jf}(\xi)|^2\mathrm{d}\xi\bigg)^\frac{1}{2}\cdot \bigg(\int_{\mathbb{R}^d}|x|^2|f|^2\mathrm{d}x\bigg)^\frac{1}{2}\\
\end{aligned}
\end{equation*}
Finally, just recall a special property of Fourier transform  $\widehat{\partial_jf}(\xi)=2\pi\mathrm{i}\cdot\xi_j\hat{f}(\xi)$. 
\subsection{An overview of Wigdersons' method}
\paragraph{}In \cite{refww}, A.Wigderson and Y.Wigderson suggested a new method to give a unified version of various uncertainty principles. This method leads to a very different proof of  (\ref{Heisenberg uncertainty principle}) in dimension $1$. Comparing with the classical proof, their method did not use any special property of Fourier transform, so it can be generalized to prove analogues of uncertainty principle for other operators. In this paper, we only focus on the Fourier transform operator. 
\par Now we give a quick review of Wigdersons' proof of (\ref{Heisenberg uncertainty principle}) in dimension $1$. Let's begin with a basic inequality: $||\hat{f}||_{L^\infty}\leq||f||_{L^1}$. Replacing $f$ by $\hat{f}$, we have $||f||_{L^\infty}\leq||\hat{f}||_{L^1}$. Multiplying these two inequalities gives us 
\begin{equation}\label{Primary uncertainty principle}
	\frac{||f||_1\cdot||\hat{f}||_1}{||f||_\infty\cdot||\hat{f}||_\infty}\geq1,\text{ where }f\in\mathcal{S}(\mathbb{R}^1)\setminus\{0\}.
\end{equation}
We call (\ref{Primary uncertainty principle}) the \textit{primary uncertainty principle}.
\par Now for $1<p<\infty$, we have an obvious inequality: $||f||^p_p\leq||f||^{p-1}_\infty||f||_1$, this is just because $|f(x)|^p\leq||f||_\infty^{p-1}|f(x)|$. Multiplying both sides by $||f||_1^{p-1}$, we get
\begin{equation}\label{p-primary, before simplification}
	||f||_1^{p-1}||f||_p^p\leq||f||_\infty^{p-1}||f||_1^{p}.
\end{equation}
The inequality (\ref{p-primary, before simplification}) implies that
\begin{equation}\label{p-primary uncertainty principle}
	\frac{||f||_1}{||f||_p}\geq\bigg(\frac{||f||_1}{||f||_\infty}\bigg)^\frac{p-1}{p},\text{ and also }	\frac{||\hat{f}||_1}{||\hat{f}||_p}\geq\bigg(\frac{||\hat{f}||_1}{||\hat{f}||_\infty}\bigg)^\frac{p-1}{p}\text{ after replacing }f\text{ by }\hat{f}. 
\end{equation}
\par Define $V(f):=\int_\mathbb{R}|x|^2|f|^2$, our goal is to prove that there exists a constant $C>0$ independent of $f$, so that
\begin{equation*}
	V(f)V(\hat{f})\geq C||f||^2_2||\hat{f}||^2_2.
\end{equation*}
Now we choose $T:=\frac{1}{8}\big(\frac{||f||_1}{||f||_2}\big)^2$, by Cauchy-Schwarz inequality,
\begin{equation*}
	\int_{-T}^T|f|=\int_{-T}^T1\cdot|f|\leq\sqrt{2T}\cdot\bigg(\int_\mathbb{R}|f|^2\bigg)^\frac{1}{2}=\frac{1}{2}||f||_1.
\end{equation*}
The above inequality is equivalent to $\frac{1}{2}||f||_1\leq\int_{|x|>T}|f|$. For the part $|x|>T$, another application of Cauchy-Schwarz inequality yields that
\begin{equation}\label{Cauch-Schwarz for |x|>T}
	\begin{aligned}
	\frac{1}{2}||f||_1\leq\int_{|x|>T}|f|\leq&\bigg(\int_{|x|>T}\frac{\mathrm{d}x}{|x|^2}\bigg)^\frac{1}{2}\bigg(\int_{|x|>T}|x|^2|f|^2\mathrm{d}x\bigg)^\frac{1}{2}\\
	\leq&\sqrt{\frac{2}{T}}\sqrt{V(f)}.
	\end{aligned}
\end{equation}
After a simplification, we get
\begin{equation}\label{primary UP for V(f)/||f||^2_2}
\frac{V(f)}{||f||_2^2}\geq\frac{1}{64}\bigg(\frac{||f||_1}{||f||_2}\bigg)^4.
\end{equation}
Replacing $f$ by $\hat{f}$ and taking a multiplication, we get,  by (\ref{p-primary uncertainty principle}) and (\ref{Primary uncertainty principle}),
\begin{equation}\label{Heisenberg_new proof}
	\frac{V(f)}{||f||_2^2}\cdot\frac{V(\hat{f})}{||\hat{f}||_2^2}\geq\frac{1}{64^2}\bigg(\frac{||f||_1||\hat{f}||_1}{||f||_2||\hat{f}||_2}\bigg)^4\geq\frac{1}{64^2}\bigg(\frac{||f||_1||\hat{f}||_1}{||f||_\infty||\hat{f}||_\infty}\bigg)^2\geq\frac{1}{64^2}.
\end{equation}
\begin{rmk}
	Comparing with the original proof, this new proof seems longer. However, we should point out that, in the classical proof, we use the identity $\widehat{\partial_jf}(\xi)=2\pi\mathrm{i}\cdot\xi_j\hat{f}(\xi)$, which is a special property of Fourier transform, and not shared by other operators. 
	\par In the new proof, essentially, we only use the property that Fourier transform and its inverse are $L^1\rightarrow L^\infty$ bounded, and this boundedness property is less restrictive. So the new proof can be generalized to many other operators, for example, if a linear operator $L:\mathcal{S}(\mathbb{R})\rightarrow\mathcal{S}(\mathbb{R})$ satisfies
	\begin{equation*}
		||L||_{L^1\rightarrow L^\infty}\leq1,\ ||L^\star Lf||_{\infty}\geq||f||_\infty,
	\end{equation*}
then we also have a Heisenberg-type uncertainty principle for $L$:
\begin{equation*}
	V(f)V(Lf)\geq C||f||_2^2||\hat{f}||_2^2.
\end{equation*}
For more details, see \cite{refww}, section 4.2 and 4.3.
\end{rmk}

\begin{rmk}
Comparing with the optimal bound $\frac{1}{16\pi^2}$, in (\ref{Heisenberg_new proof}), the constant is $1/64^2$, which is much smaller than the optimal bound. It seems that this method can not be used to prove the sharp bound, since we do not use full properties of Fourier transform.
\end{rmk}

\subsection{Questions proposed by A.Wigderson and Y.Wigderson}
\paragraph{}In \cite{refww}, the authors observed that this method can not be used to prove Heisenberg uncertainty principles in higher dimension $d\geq2$ directly. This is because the integral
$$\int_{x\in\mathbb{R}^d:|x|>T}\frac{\mathrm{d}x}{|x|^2}$$ is divergent. They asked if this convergence issues represented a real limitation of the new approach, or if there was a way around them. More precisely, they asked the following questions:
\begin{qus}[\cite{refww}, page 255, Open problem]
	Is it possible to prove the Heisenberg uncertainty principle in $\mathbb{R}^d$ via the new method? That's to say, could one use this new approach to prove the existence of an absolute constant $C=C(d)>0$ independent of $f$, so that
	\begin{equation*}
		\frac{V(f)}{||f||_2^2}\frac{V(\hat{f})}{||\hat{f}||^2_2}\geq C,\text{ for all }f\in\mathcal{S}(\mathbb{R}^d)\setminus\{0\}?
	\end{equation*}
Here $V(f):=\int_{\mathbb{R}^d}|x|^2|f(x)|^2\mathrm{d}x$. \par Moreover, we already knew that the best constant $C$ is $\frac{d^2}{16\pi^2}$. Could one use this new approach to obtain the correct dependence on the dimension? That's to say, is it possible to prove the existence of an absolute constant $C^\prime$, which is independent of $f$ and $d$, so that
	\begin{equation*}
	\frac{V(f)}{||f||_2^2}\frac{V(\hat{f})}{||\hat{f}||^2_2}\geq C^\prime d^2,\text{ for all }f\in\mathcal{S}(\mathbb{R}^d)\setminus\{0\}?
\end{equation*}
\end{qus}
\subsection{Main theorems}
\paragraph{}The following theorem answers the above two questions affirmatively.
\begin{thm}\label{L2 case}
	Via Wigdersons's method, it's possible to prove that for every $f\in\mathcal{S}(\mathbb{R}^d)\setminus\{0\}$, we have
	\begin{equation*}
		\frac{V(f)}{||f||_2^2}\frac{V(\hat{f})}{||\hat{f}||^2_2}\geq C d^2,
	\end{equation*}
where $C$ is an absolute constant independent of $f$ and $d$, for example, one may take $C=\frac{1}{10^{10}}$.
\end{thm}
Now we consider a generalization of Heisenberg uncertainty principle. For $p\in(1,\infty)$, we define	
\begin{equation*}
	V_p(f):=\int_{\mathbb{R}^d}|x|^p|f(x)|^p\mathrm{d}x.
\end{equation*}
Then $V_p(f)=V(f)$ when $p=2$, and we also have a Heisenberg-type inequality for $V_p$ when $p$ is ``small". This is the following theorem: 
\begin{thm}\label{LA case}
	Suppose that $d\geq1$.
\begin{itemize}
	\item If $1<p<\frac{2d}{d-1}$, then via Wigdersons's method, it's possible to prove that there exists a constant $C=C(d,p)$, so that
	\begin{equation*}
		\frac{V_p(f)}{||f||_p^p}\frac{V_p(\hat{f})}{||\hat{f}||^p_p}\geq C(d,p), \text{ for every } f\in\mathcal{S}(\mathbb{R}^d)\setminus\{0\}.
	\end{equation*}
	Moreover, if $1<p\leq 2$ is fixed, then as $d$ goes to infinity, the growth of $C(d,p)$ is $C_1(p)\cdot d^p$, where $C_1(p)$ is a constant independent of the dimension $d$, and up to a dimension-free constant $C_1(p)$, the growth $d^p$ is sharp.
	\item If $\frac{2d}{d-1}<p<\infty$, then 
	\begin{equation*}
		\inf_{f\in\mathcal{S}(\mathbb{R}^d)\setminus\{0\}}\frac{V_p(f)}{||f||_p^p}\frac{V_p(\hat{f})}{||\hat{f}||^p_p}=0.
	\end{equation*}
	So there is no uncertainty type theorem when $p>\frac{2d}{d-1}$.
\end{itemize}
\end{thm}
\begin{rmk}
	It seems that Wigdersons' method can not be dealt with the critical case $p=\frac{2d}{d-1}$. Also, this method does not work when $p=1$ or $p=\infty$. The three cases $p=1$, $p=\frac{2d}{d-1}$ and $p=\infty$ are unknown.
\end{rmk}

\paragraph{}Our final result deals with a theorem of Cowling-Price and its generalizations. Recall Cowling-Price uncertainty principle states that if $\theta,\phi\in(0,\infty)$ and $p,q\in[1,\infty]$, then
\begin{equation*}
		\Big|\Big||x|^\theta f\Big|\Big|_p\cdot\Big|\Big||\xi|^\phi\hat{f}\Big|\Big|_q\geq C||f||^2_2,\text{ for all }f\in\mathcal{S}(\mathbb{R})
\end{equation*}
if and only if
\begin{equation*}
\theta>\frac{1}{2}-\frac{1}{p},\ \phi>\frac{1}{2}-\frac{1}{q}\quad\text{and}\quad\theta+\frac{1}{p}=\phi+\frac{1}{q}.
\end{equation*}
For more details on this theorem, see \cite{Cowling-Price}, Theorem 5.1. Cowling-Price theorem was also discussed in \cite{Folland}, Lemma 3.3 and Theorem 3.6.
\par By using Wigdersons' method, we give another proof of Cowling-Price theorem (up to the endpoints) and its generalizations in higher dimension.
\begin{thm}\label{Cowling-Price theorem, d-dimension}
	Assume that $d\geq1$. Let $1<p,q<\infty$ and $0<\theta,\phi<\infty$, then there exists a constant $C$ so that
	\begin{equation}\label{Cowling-Price inequalitty, d-dimension}
		\Big|\Big||x|^\theta f\Big|\Big|_p\cdot\Big|\Big||\xi|^\phi\hat{f}\Big|\Big|_q\geq C||f||^2_2
	\end{equation}
	for all $f\in\mathcal{S}(\mathbb{R}^d)$, if and only if
	\begin{equation}\label{Cowling-Price d_dimension condition}
		\frac{\theta}{d}>\frac{1}{2}-\frac{1}{p},\quad\frac{\phi}{d}>\frac{1}{2}-\frac{1}{q}\quad \text{and}\quad\frac{1}{q}+\frac{\phi}{d}=\frac{1}{p}+\frac{\theta}{d}.
	\end{equation}
\end{thm}
Take $d=1$, we get the Cowling-Price theorem except for two endpoints.
\paragraph{}The article is organized as follows:
\begin{itemize}
	\item In section \ref{Proof of Theorem L2 case}, we prove Theorem \ref{L2 case}. Recall that, when $d=1$, we use integral $\int_{|x|>T}1/|x|^2$ in (\ref{Cauch-Schwarz for |x|>T}). In $\mathbb{R}^d$, we will use the integral $\int_{|x|>T}1/|x|^{d+1}$, with a special choice of some parameters, see (\ref{choice of parameters in L2 case}). In the study of the behavior of the constant $C$ as $d$ goes to infinity, we will use an asymptotic formula of Gamma function frequently:
	\begin{equation*}
		\Gamma(x)\sim\sqrt{\frac{2\pi}{x}}\Big(\frac{x}{e}\Big)^x,\text{ when }x\rightarrow\infty.
	\end{equation*}
	\item In section \ref{proof of theorem L^A case}, we prove Theorem \ref{LA case}. The proof is essentially the same with the $L^2$ case, and we only need to change the choice of some parameters. In $L^2(\mathbb{R}^d)$ case, we use the property that $\int_{x\in\mathbb{R}^d:|x|\geq1}1/|x|^{d+1}<\infty$, and in $L^p(\mathbb{R}^d)$ case we will use the integral $\int_{x\in\mathbb{R}^d:|x|\geq1}1/|x|^{d+\epsilon}<\infty$, where $\epsilon>0$ is a parameter, see Proposition \ref{choice of parameters}.
	\item In section \ref{Proof of Cowling-Price}, we prove Theorem \ref{Cowling-Price theorem, d-dimension}. We will see that the real difficulties are the first two inequalities in (\ref{Cowling-Price d_dimension condition}). The conditions (\ref{Cowling-Price d_dimension condition}) imply the inequality (\ref{Cowling-Price inequalitty, d-dimension}) is parallel to the proof of Theorem \ref{LA case}. For another direction, there are two different cases, so we should consider them separately, see section \ref{Proof of Cowling-Price theorem:  Only if part}. In both cases, assuming that (\ref{Cowling-Price d_dimension condition}) fails, we will construct a function so that (\ref{Cowling-Price inequalitty, d-dimension}) does not hold. 
\end{itemize}
\section{Proof of Theorem \ref{L2 case}}\label{Proof of Theorem L2 case}
\paragraph{}In (\ref{primary UP for V(f)/||f||^2_2}), we proved that when $d=1$:
\begin{equation*}
	\frac{V(f)}{||f||_2^2}\geq\frac{1}{64}\bigg(\frac{||f||_1}{||f||_2}\bigg)^4.
\end{equation*}
With this inequalities in mind, now we try to prove an inequality for general $d$ of the following form:
\begin{equation*}
	\frac{V(f)}{||f||_{L^2(\mathbb{R}^d)}^2}\geq C(d)\bigg(\frac{||f||_{L^a(\mathbb{R}^d)}}{||f||_{L^2(\mathbb{R}^d)}}\bigg)^\kappa,
\end{equation*}
where $a$ and $\kappa$ are constants to be determined. Also, we should take the choice of $T$ into consideration.

\subsection{Choice of some parameters \& Sketch of the proof}
\paragraph{}We choose $a,r,s\geq1$ as following:
\begin{equation}\label{choice of parameters in L2 case}
	a=\frac{2(d+1)}{d+3},\ r=\frac{d+3}{d+1},\ s=\frac{d+3}{2}.
\end{equation}
Notice that $1\leq a<2$ and $1/r+1/s=1$. Also, define 
\begin{equation*}
	T:=c_d\cdot\bigg(\frac{||f||_a}{||f||_2}\bigg)^\frac{d+1}{d},
\end{equation*}
where $c_d$ is a constant to be determined later, and it's independent of $f$, see (\ref{choice of c_d}).
\paragraph{} Our goal is to prove, for every $f\in\mathcal{S}(\mathbb{R}^d)\setminus\{0\}$,
\begin{equation}\label{primary UP for V(f)/||f||^2_2, in dimension d}
	\begin{aligned}
		\frac{V(f)}{||f||_2^2}\geq\frac{1}{4}\bigg(\frac{c_d^2}{\omega_{d-1}^2}\bigg)^\frac{1}{d+1}\bigg(\frac{||f||_a}{||f||_2}\bigg)^{(2+\frac{2}{d})},
	\end{aligned}
\end{equation}
where $\omega_{d-1}$ is the area of $(d-1)$-sphere $\mathbb{S}^{d-1}$. We will prove this inequality in the next subsection.
\paragraph{} Assuming (\ref{primary UP for V(f)/||f||^2_2, in dimension d}), then replacing $f$ by $\hat{f}$, we have
\begin{equation*}
	\begin{aligned}
		\frac{V(\hat{f})}{||\hat{f}||_2^2}\geq\frac{1}{4}\bigg(\frac{c_d^2}{\omega_{d-1}^2}\bigg)^\frac{1}{d+1}\bigg(\frac{||\hat{f}||_a}{||\hat{f}||_2}\bigg)^{(2+\frac{2}{d})}.
	\end{aligned}
\end{equation*}
Taking a multiplication we get
\begin{equation*}
	\begin{aligned}
		\frac{V(f)}{||f||_2^2}\cdot\frac{V(\hat{f})}{||\hat{f}||_2^2}\geq\frac{1}{16}\bigg(\frac{c_d^2}{\omega_{d-1}^2}\bigg)^\frac{2}{d+1}\bigg(\frac{||f||_a||\hat{f}||_a}{||f||_2||\hat{f}||_2}\bigg)^{(2+\frac{2}{d})}.
	\end{aligned}
\end{equation*}
Let $a^\prime$ be the conjugate exponent of $a$. H\"older inequality implies that
\begin{equation*}
	||\hat{f}||_2=||f||_2\leq||f||^\frac{1}{2}_a||f||^\frac{1}{2}_{a^\prime}.
\end{equation*}
Also, when $1\leq a<2$, Hausdorff-Young inequality yields that $||\hat{f}||_{a^\prime}\leq||f||_a$, replacing $f$ by $\hat{f}$, we have $||f||_{a^\prime}\leq||\hat{f}||_a$, and
\begin{equation}\label{L^2 estimate for f and its fourier transform}
	\frac{||f||_a||\hat{f}||_a}{||f||_2||\hat{f}||_2}\geq\frac{||f||_a||\hat{f}||_a}{||f||_a||f||_{a^\prime}}=\frac{||\hat{f}||_a}{||f||_{a^\prime}}\geq1.
\end{equation}
Therefore,
\begin{equation*}
	\begin{aligned}
		\frac{V(f)}{||f||_2^2}\cdot\frac{V(\hat{f})}{||\hat{f}||_2^2}\geq\frac{1}{16}\bigg(\frac{c_d^2}{\omega_{d-1}^2}\bigg)^\frac{2}{d+1}.
	\end{aligned}
\end{equation*}
\paragraph{}Now only need to study the growth of the  quotient $\big(\frac{c_d}{\omega_{d-1}}\big)^\frac{2}{d+1}$, when $d$ goes to infinity. In subsection \ref{asymp of quotient, L^2 case}, we will prove that
\begin{equation*}
	\bigg(\frac{c_d}{\omega_{d-1}}\bigg)^\frac{2}{d+1}>\frac{1}{C}d,
\end{equation*}
where $C$ is a large constant independent of $d$, for example $C=100000$, and we eventually get
\begin{equation*}
	\begin{aligned}
		\frac{V(f)}{||f||_2^2}\cdot\frac{V(\hat{f})}{||\hat{f}||_2^2}\geq\frac{1}{C^\prime}d^2.
	\end{aligned}
\end{equation*}
Here one may take $C^\prime=10^{10}$.

\subsection{Proof of (\ref{primary UP for V(f)/||f||^2_2, in dimension d})}To prove (\ref{primary UP for V(f)/||f||^2_2, in dimension d}), we use H\"older inequality for the pair $r$ and $s$, notice that $ar=2$, $as=d+1$ and $1/s=(2-a)/2$ by our choice.
\begin{equation*}
	\begin{aligned}
	\int_{|x|\leq T}|f|^a\leq&\bigg(\int_{|x|\leq T}1^s\mathrm{d}x\bigg)^\frac{1}{s}\cdot \bigg(\int_{|x|\leq T}|f|^{ar}\bigg)^\frac{1}{r}\\
	=&(T^dv_d)^{\frac{1}{s}}\bigg(\int_{|x|\leq T}|f|^{2}\bigg)^\frac{a}{2}\\
	\leq&(T^dv_d)^{\frac{2-a}{2}}||f||_2^{a}
	\end{aligned},
\end{equation*}
where $v_d$ is the volume of the unite ball in $\mathbb{R}^d$.
\par The term $(T^dv_d)^{\frac{2-a}{2}}$ is 
\begin{equation*}
	(T^dv_d)^{\frac{2-a}{2}}=(c_d^dv_d)^{\frac{2-a}{2}}\bigg(\frac{||f||^{d+1}_a}{||f||^{d+1}_2}\bigg)^{\frac{2-a}{2}}
\end{equation*}
Notice that $\frac{(d+1)(2-a)}{2}=a$, so
\begin{equation*}
	\bigg(\frac{||f||^{d+1}_a}{||f||^{d+1}_2}\bigg)^{\frac{2-a}{2}}=\bigg(\frac{||f||_a}{||f||_2}\bigg)^a
\end{equation*}
Now we choose $c_d$ so that 
\begin{equation}\label{choice of c_d}
	(c_d^dv_d)^{\frac{2-a}{2}}=1/2.
\end{equation}
Particularly, the constant $c_d$ depends only on dimension $d$.
For this $c_d$, we have
\begin{equation*}
		\begin{aligned}
		\int_{|x|\leq T}|f|^a\leq\frac{1}{2}||f||_a^a,
	\end{aligned}
\end{equation*}
which is equivalent to $\int_{|x|\geq T}|f|^a\geq\frac{1}{2}||f||_a^a$.

\paragraph{}Another application of H\"older inequality gives
\begin{equation*}
	\begin{aligned}
		\int_{|x|\geq T}|f|^a=&\int_{|x|\geq T}\frac{1}{|x|^a}|x|^a|f(x)|^a\\
		\leq&\bigg(\int_{|x|\geq T}\frac{1}{|x|^{as}}\bigg)^\frac{1}{s}\cdot\bigg(\int_{|x|\geq T}|x|^{ar}|f(x)|^{ar}\bigg)^\frac{1}{r}\\
		=&\bigg(\int_{|x|\geq T}\frac{1}{|x|^{d+1}}\bigg)^\frac{1}{s}\cdot\bigg(\int_{|x|\geq T}|x|^{2}|f(x)|^{2}\bigg)^\frac{1}{r}\\
		=&\bigg(\frac{1}{T}\int_{|x|\geq 1}\frac{1}{|x|^{d+1}}\bigg)^\frac{1}{s}\cdot V(f)^\frac{1}{r}\\
	\end{aligned}
\end{equation*}
An easy calculation shows that $\int_{|x|\geq 1}\frac{1}{|x|^{d+1}}=\omega_{d-1}$. Therefore

\begin{equation*}
	\begin{aligned}
			\frac{1}{2}||f||_a^a\leq\int_{|x|\geq T}|f|^a\leq&\bigg(\frac{\omega_{d-1}}{T}\bigg)^\frac{2}{d+3}\cdot V(f)^\frac{d+1}{d+3}.\\
	\end{aligned}
\end{equation*}
After a simplification, we obtain
\begin{equation*}
	\begin{aligned}
		\frac{1}{2^{d+3}}\frac{||f||_a^{a(d+3)+\frac{2(d+1)}{d}}}{||f||_2^{\frac{2(d+1)}{d}+2(d+1)}}\leq\frac{\omega_{d-1}^2}{c_d^2}\bigg(\frac{V(f)}{||f||_2^2}\bigg)^{d+1}
	\end{aligned}
\end{equation*}
Notice that	$a(d+3)=2(d+1)$. So
\begin{equation*}
	\begin{aligned}
		\frac{\omega_{d-1}^2}{c_d^2}\bigg(\frac{V(f)}{||f||_2^2}\bigg)^{d+1}\geq	\frac{1}{2^{d+3}}\frac{||f||_a^{a(d+3)+\frac{2(d+1)}{d}}}{||f||_2^{\frac{2(d+1)}{d}+2(d+1)}}=\frac{1}{2^{d+3}}\bigg(\frac{||f||_a}{||f||_2}\bigg)^{(d+1)(2+\frac{2}{d})},
	\end{aligned}
\end{equation*}
taking a $\frac{1}{d+1}$-power we get
\begin{equation*}
	\begin{aligned}
		\frac{V(f)}{||f||_2^2}\geq\bigg(\frac{c_d^2}{\omega_{d-1}^2}\bigg)^\frac{1}{d+1}\frac{1}{2^{\frac{d+3}{d+1}}}\bigg(\frac{||f||_a}{||f||_2}\bigg)^{(2+\frac{2}{d})}\geq\frac{1}{4}\bigg(\frac{c_d^2}{\omega_{d-1}^2}\bigg)^\frac{1}{d+1}\bigg(\frac{||f||_a}{||f||_2}\bigg)^{(2+\frac{2}{d})}.
	\end{aligned}
\end{equation*}

\subsection{Asymptotic behavior of the quotient}\label{asymp of quotient, L^2 case}
\paragraph{}In this subsection, we study the growth of  $\big(\frac{c_d^2}{\omega_{d-1}^2}\big)^\frac{1}{d+1}$ when $d$ tends to infinity. Recall the choice of $c_d$, after a calculation we have
\begin{equation*}
	\begin{aligned}
		\bigg(\frac{c_d}{\omega_{d-1}}\bigg)^\frac{2}{d+1}=\frac{1}{(v_d)^\frac{2}{d(d+1)}2^\frac{d+3}{d(d+1)}\omega_{d-1}^\frac{2}{d+1}}
	\end{aligned}
\end{equation*}
\paragraph{}The first factor
\begin{equation*}
	\begin{aligned}
		\frac{1}{(v_d)^\frac{2}{d(d+1)}}=\frac{\Gamma(\frac{d}{2}+1)^\frac{2}{d(d+1)}}{\pi^\frac{1}{d+1}}\in\bigg[\frac{\Gamma(\frac{d}{2}+1)^\frac{2}{d(d+1)}}{\sqrt{\pi}},\Gamma(\frac{d}{2}+1)^\frac{2}{d(d+1)}\bigg]
	\end{aligned}
\end{equation*}
Recall the approximation of Gamma function when $x>0$ is large: $\Gamma(x)\sim\sqrt{\frac{2\pi }{x 	}}\big(\frac{x}{e}\big)^x$.
\begin{equation*}
	\begin{aligned}
\Gamma(\frac{d}{2}+1)^\frac{2}{d(d+1)}\sim&\Bigg(\sqrt{\frac{4\pi}{d+2}}\bigg(\frac{2+d}{2\mathrm{e}}\bigg)^{\frac{d}{2}+1}\Bigg)^\frac{2}{d(d+1)}\\
<&\Bigg(2\sqrt{\frac{\pi}{d+2}}\bigg(\frac{2+d}{\pi}\bigg)^{\frac{d}{2}+1}\Bigg)^\frac{2}{d(d+1)}\\
\leq&2\bigg(\frac{2+d}{\pi}\bigg)^{\frac{1}{d}}.
\end{aligned}
\end{equation*}
Notice that if $d\geq1$, then the following two inequalities hold: $\frac{2+d}{\pi}\leq d$, and $d^\frac{1}{d}\leq \mathrm{e}$. So we have
\begin{equation}\label{Asymptotic behavior of Gamma(1+d/2)}
	1\leq\Gamma(\frac{d}{2}+1)^\frac{2}{d(d+1)}\leq 2\mathrm{e}.
\end{equation}
The lower bound is trivial because $\Gamma(x)\geq1$ when $x\geq1$. Actually, we cheated a little bit, because the asymptotic formula $\Gamma(x)\sim\sqrt{\frac{2\pi }{x 	}}\big(\frac{x}{e}\big)^x$ holds only when $x$ is large. The inequality (\ref{Asymptotic behavior of Gamma(1+d/2)}) is true only when $d$ is large. However, this is really a minor problem, since we're only interested in the behavior of the constant when $d$ is large, and we may assume $d$ is already large at the very beginning.
\par These bounds implies that the first factor $\frac{1}{(v_d)^\frac{2}{d(d+1)}}$ lies in the interval $[\frac{1}{\sqrt{\pi}},2\mathrm{e}]$, which does not give a decay or growth on dimension $d$. Like the first factor, the 2nd factor also makes no contribution on the growth because
\begin{equation*}
	\frac{1}{4}\leq\frac{1}{2^\frac{d+3}{d(d+1)}}\leq1.
\end{equation*}
\paragraph{}Recall the surface area $\omega_{d-1}=\frac{2\pi^\frac{d}{2}}{\Gamma(\frac{d}{2})}$. So for the 3rd factor, we have
\begin{equation*}
	\begin{aligned}
		\frac{1}{\omega_{d-1}^\frac{2}{d+1}}=&\frac{\Gamma(\frac{d}{2})^\frac{2}{d+1}}{(2\pi^\frac{d}{2})^\frac{2}{d+1}}>\frac{\Gamma(\frac{d}{2})^\frac{2}{d+1}}{10}.
	\end{aligned}
\end{equation*}
Combining all three factors yields
\begin{equation*}
	\frac{1}{(v_d)^\frac{2}{d(d+1)}2^\frac{d+3}{d(d+1)}\omega_{d-1}^\frac{2}{d+1}}>\frac{1}{100}\Gamma(\frac{d}{2})^\frac{2}{d+1}.
\end{equation*}
Finally
\begin{equation*}
	\begin{aligned}
		\Gamma(\frac{d}{2})^\frac{2}{d+1}\sim\Bigg(\sqrt{\frac{4\pi}{d}}\bigg(\frac{d}{2e}\bigg)^\frac{d}{2}\Bigg)^\frac{2}{d+1}>\bigg(\frac{d}{2e}\bigg)^{1-\frac{2}{d+1}}=\frac{d}{2e}\bigg(\frac{d}{2e}\bigg)^{-\frac{2}{d+1}}>\frac{d}{2\mathrm{e}^3}.
	\end{aligned}
\end{equation*}
So
\begin{equation*}
	\bigg(\frac{c_d}{\omega_{d-1}}\bigg)^\frac{2}{d+1}>\frac{1}{10^5}d,
\end{equation*}
which implies that
\begin{equation*}
	\begin{aligned}
		\frac{V(f)}{||f||_2^2}\cdot\frac{V(\hat{f})}{||\hat{f}||_2^2}\geq\bigg(\frac{c^2_d}{\omega^2_{d-1}}\bigg)^\frac{2}{d+1}\geq\frac{1}{10^{10}}d^2.
	\end{aligned}
\end{equation*}

\section{$L^p(\mathbb{R}^d)$ case}\label{proof of theorem L^A case}
\paragraph{}Now we consider the general case $L^p(\mathbb{R}^d)$. Parallel to $L^2$, now we try to prove an inequality of the following form:
\begin{equation}\label{L_A estimate for a single function}
	\frac{V_p(f)}{||f||_{L^p(\mathbb{R}^d)}^p}\geq C(d)\bigg(\frac{||f||_{L^a(\mathbb{R}^d)}}{||f||_{L^p(\mathbb{R}^d)}}\bigg)^\kappa,
\end{equation}
where $a,\kappa$ are constants to be determined. Suppose that (\ref{L_A estimate for a single function}) holds, replacing $f$ by $\hat{f}$ and take a product, 
\begin{equation*}\label{L_A estimate for a function with its fourier transform}
		\frac{V_p(f)}{||f||_p^p}\cdot\frac{V_p(\hat{f})}{||\hat{f}||_p^p}\geq C(d)^2\Bigg(\frac{||f||_{a}||\hat{f}||_{a}}{||f||_p||\hat{f}||_p}\Bigg)^\kappa.
\end{equation*}
Recall (\ref{L^2 estimate for f and its fourier transform}), we proved $\frac{||f||_a||\hat{f}||_a}{||f||_2||\hat{f}||_2}\geq1$, when $a=2(d+1)/(d+3)$. However, here we should prove
\begin{equation*}
	\frac{||f||_{a}||\hat{f}||_{a}}{||f||_p||\hat{f}||_p}\geq1,
\end{equation*}
where $a$ is a constant to be chosen. This is already known:
\begin{thm}[\cite{refT}, Theorem 1.3.] \label{L^A primary UP, higher dimension}
	If $1<a<p<\infty$ and $\frac{1}{a}+\frac{1}{p}\geq1$, then
\begin{equation*}
	\frac{||f||_{a}||\hat{f}||_{a}}{||f||_p||\hat{f}||_p}\geq1,\text{ for every }f\in\mathcal{S}(\mathbb{R}^d)\setminus\{0\}.
\end{equation*}
\end{thm}
This theorem was proved in \cite{refT} for $d=1$, but it can be generalized to general $d$ without any difficulties. We only used H\"older inequality and Hausdorff-Young inequality in the proof of Theorem \ref{L^A primary UP, higher dimension}, and both inequalities are true for general $d$.
\subsection{Choice of some parameters}
The following proposition and (\ref{choice of parameters in L^A case}) are substitutes of (\ref{choice of parameters in L2 case}) in $L^p$ case.
\begin{prop}\label{choice of parameters}
	For every $p\in(1,\frac{2d}{d-1})$, one can always find an $\epsilon>0$ so that
	\begin{equation*}
		\frac{(d+\epsilon)p}{d+\epsilon+p}>1,\text{ and } p\leq\frac{2(d+\epsilon)}{d-1+\epsilon}.
	\end{equation*}
\end{prop}
\begin{proof}
These two inequalities are equivalent to
\begin{equation*}
	\epsilon>\frac{d+p-dp}{p-1},\text{ and }(p-2)\epsilon\leq2d-p(d-1).
\end{equation*}
Notice that $2d-p(d-1)>0$ always holds. There are two cases.
\begin{itemize}
	\item When $p\in(1,2]$: The second inequality $(p-2)\epsilon\leq2d-p(d-1)$ is trivial, because $(p-2)\epsilon$ is negative while $2d-p(d-1)>0$. We only need to choose an $\epsilon>0$ satisfies the first inequality, for example $\epsilon=p/(p-1)$.
	\item When $p\in(2,\frac{2d}{d-1})$: Now we need to prove the existence of $\delta>0$ so that
	\begin{equation*}
		\frac{d+p-dp}{p-1}<\delta\leq\frac{2d-p(d-1)}{p-2}.
	\end{equation*}
It suffices to prove that the interval $(\frac{d+p-dp}{p-1},\frac{2d-p(d-1)}{p-2}]$ is non-empty, and the right endpoint $\frac{2d-p(d-1)}{p-2}>0$. These are just direct calculations.
\end{itemize}
\end{proof}
Now we choose, for $\epsilon$ in Proposition \ref{choice of parameters},
\begin{equation}\label{choice of parameters in L^A case}
	\begin{aligned}
	a&:=\frac{p(d+\epsilon)}{d+\epsilon+p},\\
	r&:=\frac{p}{a}=\frac{d+\epsilon+p}{d+\epsilon},\ s=\frac{d+\epsilon+p}{p},\\
	T&:=c_d\bigg(\frac{||f||_a}{||f||_p}\bigg)^\frac{as}{d}=c_d\bigg(\frac{||f||_a}{||f||_p}\bigg)^\frac{d+\epsilon}{d},
\end{aligned}
\end{equation}
where $c_d$ is a constant independent of function $f$ to be determined later, see (\ref{choice of c_d for L^A}). Notice that $1/r+1/s=1$, and $1<r,s<\infty$. By Proposition \ref{choice of parameters}, it's easy to verify
\begin{equation*}
	1<a<p,\text{ and }\frac{1}{a}+\frac{1}{p}\geq1.
\end{equation*}
So the pair $(a,p)$ verifies the hypothesis in Theorem \ref{L^A primary UP, higher dimension}. In fact, the inequality $a<p$ is trivial, and $a>1$ is equivalent to $(d+\epsilon)p/(d+\epsilon+p)>1$. Also, the second inequality $1/a+1/p\geq1$ is equivalent to $p\leq2(d+\epsilon)/(d-1+\epsilon)$.
\paragraph{}Similar to the case $L^2$, we use H\"older inequality for $p$ and $q$:
\begin{equation*}
	\begin{aligned}
		\int_{|x|\leq T}|f|^a\leq&\bigg(\int_{|x|\leq T}1\mathrm{d}x\bigg)^\frac{1}{s}\cdot \bigg(\int_{|x|\leq T}|f|^{ar}\bigg)^\frac{1}{r}\\
		\leq&(T^dv_d)^{\frac{1}{s}}||f||_p^{\frac{p}{r}}\\
		\leq&(v_dc_d^d)^\frac{1}{s}\bigg(\frac{||f||_a}{||f||_p}\bigg)^a||f||_p^{\frac{p}{r}}\\
		=&(v_dc_d^d)^\frac{1}{s}||f||_a^a.
	\end{aligned}
\end{equation*}
Choose $c_d$ so that
\begin{equation}\label{choice of c_d for L^A}
	(v_dc^d_d)^\frac{1}{s}=1/2,
\end{equation}
and we have
\begin{equation*}
	\int_{|x|>T}|f|^a\geq\frac{1}{2}\int_{\mathbb{R}^d}|f|^a.
\end{equation*}
\paragraph{}For the integral $\int_{|x|>T}|f|^a$, again, we use H\"older inequality. In $L^2(\mathbb{R}^d)$ case, we used $\int_{|x|>T}1/|x|^{d+1}$. Now we replace it by $\int_{|x|>T}1/|x|^{d+\epsilon}$. 
\begin{equation*}
	\begin{aligned}
			\int_{|x|\geq T}|f|^a=&\int_{|x|\geq T}\frac{1}{|x|^a}|x|^a|f(x)|^a\\
			\leq&\bigg(\int_{|x|\geq T}\frac{1}{|x|^{as}}\bigg)^\frac{1}{s}\cdot\bigg(\int_{|x|\geq T}|x|^{ar}|f(x)|^{ar}\bigg)^\frac{1}{r}\\
			=&\bigg(\int_{|x|\geq T}\frac{1}{|x|^{d+\epsilon}}\bigg)^\frac{1}{s}\cdot\bigg(\int_{|x|\geq T}|x|^{p}|f(x)|^{p}\bigg)^\frac{1}{r}\\
			\leq&\bigg(\frac{\omega_{d-1}}{\epsilon}\bigg)^\frac{1}{s}\bigg(\frac{1}{T}\bigg)^\frac{\epsilon}{s}V_p(f)^\frac{1}{r}.
	\end{aligned}
\end{equation*}
We have
\begin{equation*}
\bigg(\frac{1}{T}\bigg)^\frac{\epsilon}{s}=\bigg(\frac{1}{c_d^\epsilon}\bigg)^\frac{1}{s}\bigg(\frac{||f||_p}{||f||_a}\bigg)^\frac{a\epsilon}{d}
\end{equation*}
So
\begin{equation*}
	\begin{aligned}
	\frac{1}{2}||f||_a^a\leq\int_{|x|\geq T}|f|^a\leq&\bigg(\frac{\omega_{d-1}}{\epsilon c_d^\epsilon}\bigg)^\frac{1}{s}\bigg(\frac{||f||_p}{||f||_a}\bigg)^\frac{a\epsilon}{d}V_p(f)^\frac{1}{r}.
\end{aligned}
\end{equation*}
This is equivalent to
\begin{equation*}
	\begin{aligned}
		\frac{||f||_a^{a+\frac{a\epsilon}{d}}}{||f||_p^{\frac{a\epsilon}{d}+\frac{p}{r}}}\leq2\bigg(\frac{\omega_{d-1}}{\epsilon c_d^\epsilon}\bigg)^\frac{1}{s}\bigg(\frac{V_p(f)}{||f||_p^{p}}\bigg)^\frac{1}{r}.
	\end{aligned}
\end{equation*}
By the choice of $p$ and $a$, we have $p/r=a$. Therefore
\begin{equation*}
	\begin{aligned}
\frac{V_p(f)}{||f||_p^{p}}\geq \frac{1}{2}\bigg(\frac{\epsilon c_d^\epsilon}{2\omega_{d-1}}\bigg)^{r-1}	\bigg(\frac{||f||_a}{||f||_p}\bigg)^{p(1+\frac{\epsilon}{d})}.
	\end{aligned}
\end{equation*}
Replace $f$ by $\hat{f}$,
\begin{equation*}
	\begin{aligned}
		\frac{V_p(\hat{f})}{||\hat{f}||_p^{p}}\geq \frac{1}{2}\bigg(\frac{\epsilon c_d^\epsilon}{2\omega_{d-1}}\bigg)^{r-1}	\bigg(\frac{||\hat{f}||_a}{||\hat{f}||_p}\bigg)^{p(1+\frac{\epsilon}{d})}.
	\end{aligned}
\end{equation*}
Taking a product and use Theorem \ref{L^A primary UP, higher dimension}, we get
\begin{equation*}
	\begin{aligned}
		\frac{V_p(\hat{f})\cdot V_p(f)}{||\hat{f}||_p^{p}\cdot ||f||_p^{p}}\geq\frac{1}{4}\bigg(\frac{\epsilon c_d^\epsilon}{2\omega_{d-1}}\bigg)^{2(r-1)}\bigg(\frac{||f||_a\cdot||\hat{f}||_a}{||f||_p\cdot||\hat{f}||_p}\bigg)^{p(1+\frac{\epsilon}{d})}\geq\frac{1}{4}\bigg(\frac{\epsilon c_d^\epsilon}{2\omega_{d-1}}\bigg)^{2(r-1)}.
	\end{aligned}
\end{equation*}
\subsection{Asymptotic behavior of the quotient}
\paragraph{}In this subsection, we study the growth of the quotient  $\big(\frac{\epsilon c_d^\epsilon}{2\omega_{d-1}}\big)^{r-1}$ when $d$ tends to infinity and $p$ fixed. Notice that if $p$ is fixed and $1<p<2d/(d-1)$ holds for all $d\geq1$, then we must have $1<p\leq2$. The case $p=2$ has been proved. Therefore, in this subsection, we consider a fixed number $p\in(1,2)$. Moreover, once $p\in(1,2)$ is fixed, the choice of $\epsilon$ in Proposition \ref{choice of parameters} can be made independent of $d$, for example, we may take $\epsilon=p/(p-1)>1$. By (\ref{choice of c_d for L^A}), 
\begin{equation*}
	\begin{aligned}
	\bigg(\frac{\epsilon c_d^\epsilon}{2\omega_{d-1}}\bigg)^{r-1}=\bigg(\frac{\epsilon}{2^{\frac{s\epsilon}{d}+1}}\bigg)^{r-1}\bigg(\frac{1}{\omega_{d-1}v_d^{\frac{\epsilon}{d}}}\bigg)^{r-1}
\end{aligned}
\end{equation*}
Recall $r=1+\frac{p}{d+\epsilon}$, for the second factor, we have
\begin{equation*}
	\begin{aligned}
	\bigg(\frac{1}{\omega_{d-1}v_d^{\frac{\epsilon}{d}}}\bigg)^{r-1}&=\frac{1}{\omega_{d-1}^{\frac{p}{d+\epsilon}}v_d^{\frac{\epsilon p}{d(d+\epsilon)}}}
\end{aligned}
\end{equation*}
\paragraph{Asymptotic behavior of $\big(\epsilon/2^{\frac{s\epsilon}{d}+1}\big)^{r-1}$:} Recall that $\epsilon$ is independent of $d$, we have
\begin{equation*}
	\bigg(\frac{\epsilon}{2^{\frac{s\epsilon}{d}+1}}\bigg)^{r-1}=\frac{\epsilon^{\frac{p}{d+\epsilon}}}{2^{\frac{\epsilon r}{d}+(r-1)}}\sim\frac{1}{2^{(\frac{\epsilon}{d}+1)r}}=\frac{1}{2^{\frac{d+\epsilon+p}{d}}}\sim1.
\end{equation*}

\paragraph{Asymptotic behavior of $v_d^{\frac{\epsilon p}{d(d+\epsilon)}}$:}We have
\begin{equation*}
	\frac{1}{v_d^{\frac{\epsilon p}{d(d+\epsilon)}}}=\frac{\Gamma\Big(\frac{d}{2}+1\Big)^\frac{\epsilon p}{d(d+\epsilon)}}{\pi^\frac{p}{2(d+\epsilon)}}\sim \Gamma\Big(\frac{d}{2}+1\Big)^\frac{\epsilon p}{d(d+\epsilon)}\sim\bigg(\sqrt{\frac{4\pi}{d+2}}\Big(\frac{d+2}{2\mathrm{e}}\Big)^{\frac{d}{2}+1}\bigg)^\frac{\epsilon p}{d(d+\epsilon)}\sim1.
\end{equation*}

\paragraph{Asymptotic behavior of $\omega_{d-1}^{\frac{p}{d+\epsilon}}$:}We have
\begin{equation*}
	\begin{aligned}
		\frac{1}{\omega_{d-1}^{\frac{p}{d+\epsilon}}}=\frac{\Gamma(d/2)^\frac{p}{d+\epsilon}}{(2\pi^{d/2})^\frac{p}{d+\epsilon}}\sim\Gamma\Big(\frac{d}{2}\Big)^\frac{p}{d+\epsilon}\sim\bigg(\sqrt{\frac{4\pi}{d}}\big(\frac{d}{2\mathrm{e}}\big)^{\frac{d}{2}}\bigg)^\frac{p}{d+\epsilon}.
	\end{aligned}
\end{equation*}
The following inequalities holds for all $d\geq1$:
\begin{equation*}
	\sqrt{\frac{2\mathrm{e}}{d}}\big(\frac{d}{2\mathrm{e}}\big)^{\frac{d}{2}}<\sqrt{\frac{4\pi}{d}}\big(\frac{d}{2\mathrm{e}}\big)^{\frac{d}{2}}<2\sqrt{\frac{2\mathrm{e}}{d}}\big(\frac{d}{2\mathrm{e}}\big)^{\frac{d}{2}}.
\end{equation*}
Therefore
\begin{equation*}
	\bigg(\sqrt{\frac{4\pi}{d}}\big(\frac{d}{2\mathrm{e}}\big)^{\frac{d}{2}}\bigg)^\frac{p}{d+\epsilon}\sim\bigg(\sqrt{\frac{2\mathrm{e}}{d}}\big(\frac{d}{2\mathrm{e}}\big)^{\frac{d}{2}}\bigg)^\frac{p}{d+\epsilon}=\bigg(\frac{d}{2\mathrm{e}}\bigg)^{\frac{p(d-1)}{2(d+\epsilon)}}\sim d^\frac{p}{2}.
\end{equation*}

\paragraph{Sharpness of the growth $d^p$:} Taking Gaussian $g(x)=e^{-\pi|x|^2}$, a direct calculation shows that 
\begin{equation*}
	||g||_p^p=\frac{1}{p^{\frac{d}{2}}},\text{ and }V_p(g)=\frac{\omega_{d-1}}{2(\pi p)^\frac{p+d}{2}}\Gamma\Big(\frac{p+d}{2}\Big).
\end{equation*}
Therefore,
	\begin{equation*}
	\frac{V_p(g)}{||g||_p^p}\frac{V_p(\hat{g})}{||\hat{g}||^p_p}=\frac{1}{(\pi p)^p}\frac{\Gamma\Big(\frac{p+d}{2}\Big)^2}{\Gamma\Big(\frac{d}{2}\Big)^2}\sim d^p,
\end{equation*}
which shows that up to a dimension-free constant, the growth $d^p$ is optimal.
\subsection{Sharpness of the range $p$}
To prove the sharpness of the range, we need to show that when $p\in(\frac{2d}{d-1},\infty)$,
\begin{equation*}
	\inf_{f\in\mathcal{S}(\mathbb{R}^d)\setminus\{0\}}\frac{V_p(f)}{||f||_p^p}\frac{V_p(\hat{f})}{||\hat{f}||^p_p}=0.
\end{equation*}
We use the functions
\begin{equation*}
	g_c=c^{-\frac{d}{2}}\mathrm{e}^{-\pi\frac{|x|^2}{c^2}}+c^\frac{d}{2}\mathrm{e}^{-\pi c^2|x|^2},\ c>0.
\end{equation*}
Notice that $g_c=\widehat{g_c}$, and these functions already appeared in \cite{refww} when $d=1$. We estimate the $L^p$ norm of $g_c$ and $|x|\cdot|g_c|$
\begin{equation*}
\begin{aligned}
	\int|g_c|^p\geq& c^{-\frac{dp}{2}}\int\mathrm{e}^{-\pi\frac{p|x|^2}{c^2}}\mathrm{d}x+c^{\frac{dp}{2}}\int\mathrm{e}^{-\pi pc^2|x|^2}\mathrm{d}x\\
	=&c^{-\frac{dp}{2}}\cdot \frac{c^d}{p^\frac{d}{2}}+c^{\frac{dp}{2}}\cdot \frac{1}{p^\frac{d}{2}c^d}\\
	=&\frac{c^{d-\frac{dp}{2}}+c^{\frac{dp}{2}-d}}{p^\frac{d}{2}}.
\end{aligned}
\end{equation*}
For $|x|g_c$, we have
\begin{equation*}
	\begin{aligned}
		\int|x|^p|g_c|^p&\leq2^{p-1}\int|x|^p c^{-\frac{dp}{2}}\mathrm{e}^{-\pi\frac{p|x|^2}{c^2}}+|x|^pc^\frac{dp}{2}\mathrm{e}^{-\pi pc^2|x|^2}\\
		&\leq\frac{2^{p-1}\omega_{d-1}\Gamma(\frac{d+p}{2})}{2(\pi
	p)^\frac{d+p}{2}}\bigg(c^{d+p-\frac{dp}{2}}+c^{-(d+p)+\frac{dp}{2}}\bigg).
	\end{aligned}
\end{equation*}
So
\begin{equation*}
	\frac{\int|x|^p|g_c|^p}{\int|g_c|^p}\lesssim\frac{c^{d+p-\frac{dp}{2}}+c^{-(d+p)+\frac{dp}{2}}}{c^{d-\frac{dp}{2}}+c^{\frac{dp}{2}-d}}=:h(c),
\end{equation*}
where the implicit constant in $\lesssim$ is independent of $c$. Let $t:=c^{\frac{dp}{2}-d}$, then
\begin{equation*}
	\frac{c^{d+p-\frac{dp}{2}}+c^{-(d+p)+\frac{dp}{2}}}{c^{d-\frac{dp}{2}}+c^{\frac{dp}{2}-d}}=\frac{t^{\frac{2}{(p-2)d}(d+p-\frac{dp}{2})}+t^{-\frac{2}{(p-2)d}(d+p-\frac{dp}{2})}}{t+\frac{1}{t}}
\end{equation*}
There are two cases
\begin{itemize}
	\item The case $d+p-\frac{dp}{2}\geq0$. Notice that $p>\frac{2d}{d-1}>2$, so we have $\frac{2}{(p-2)d}(d+p-\frac{dp}{2})\geq0$. In this case, it's easy to check that $p>\frac{2d}{d-1}$ is equivalent to
	\begin{equation*}
		\frac{2}{(p-2)d}(d+p-\frac{dp}{2})<1.
	\end{equation*}

    \item The case $d+p-\frac{dp}{2}<0$, now we have $-\frac{2}{(p-2)d}(d+p-\frac{dp}{2})>0$. It is easy to check that $p>1$ implies
    \begin{equation*}
    	-\frac{2}{(p-2)d}(d+p-\frac{dp}{2})<1.
    \end{equation*}
\end{itemize}
So after a change of variable $t=c^{\frac{dp}{2}-d}$, the function $h(c)$ has the form
\begin{equation*}
	\frac{t^\alpha+\frac{1}{t^\alpha}}{t+\frac{1}{t}},\ 0\leq\alpha<1.
\end{equation*}
Let $t\rightarrow\infty$, we get the infimum which is $0$.

\section{Cowling-Price theorem in higher dimension}\label{Proof of Cowling-Price}
Firstly, we point out that the third condition $\frac{1}{q}+\frac{\phi}{d}=\frac{1}{p
}+\frac{\theta}{d}$ is an easy consequence of homogeneous consideration. If (\ref{Cowling-Price inequalitty, d-dimension}) holds for all $f\in\mathcal{S}(\mathbb{R}^d)$, then it also holds for $f(cx)$ for all $c>0$. After a simplification we get
\begin{equation*}
	\frac{c^{\phi+\frac{d}{q}}}{c^dc^{\theta+\frac{\theta}{p}}}	\Big|\Big||x|^\theta f\Big|\Big|_p\cdot\Big|\Big||\xi|^\phi\hat{f}\Big|\Big|_q\geq\frac{1}{c^d}||f||^2_2.
\end{equation*}
We see that $d/q+\phi=\theta/p+\theta$ must hold, otherwise, take $c\rightarrow0$ or $c\rightarrow\infty$ would give a contradiction. So the real difficult parts are first two inequalities in (\ref{Cowling-Price d_dimension condition}).

\subsection{Proof of Theorem \ref{Cowling-Price theorem, d-dimension}: If part}
Similar to the proof of Theorem \ref{LA case}, we should make careful choices of some parameters. The following proposition is an analogue of Proposition \ref{choice of parameters}.
\begin{prop}\label{choice of parameters: Cowling-Price}
	There exists $\delta>0$, so that
	\begin{equation*}
		\begin{aligned}
			1+\frac{d}{\phi q}-(1-\frac{1}{p})(1+\frac{d}{\phi q})\frac{d}{\theta}<\delta<\min\Big\{1+\frac{d}{\phi q},1+\frac{d}{\theta p},	1+\frac{d}{\phi q}-(\frac{1}{2}-\frac{1}{p})(1+\frac{d}{\phi q})\frac{d}{\theta}\Big\}.
		\end{aligned}
	\end{equation*}
\end{prop}
\begin{proof}
	We consider two cases: $p\geq2$ and $1<p<2$.
	\begin{itemize}
		\item When $p\geq2$, it suffices to prove that there exists a strictly positive $\delta$ so that
			\begin{equation*}
			\begin{aligned}
				1+\frac{d}{\phi q}-(1-\frac{1}{p})(1+\frac{d}{\phi q})\frac{d}{\theta}<\delta<\min\Big\{1+\frac{d}{\theta p},	1+\frac{d}{\phi q}-(\frac{1}{2}-\frac{1}{p})(1+\frac{d}{\phi q})\frac{d}{\theta}\Big\}.
			\end{aligned}
		\end{equation*}
	For this we only need to show that  
	\begin{equation*}
		\min\Big\{1+\frac{d}{\theta p},	1+\frac{d}{\phi q}-(\frac{1}{2}-\frac{1}{p})(1+\frac{d}{\phi q})\frac{d}{\theta}\Big\}>\max\Big\{0,	1+\frac{d}{\phi q}-(1-\frac{1}{p})(1+\frac{d}{\phi q})\frac{d}{\theta}\Big\}.
	\end{equation*}
We omit details and point out that
\begin{equation*}
	\begin{aligned}
		1+\frac{d}{\theta p}>1+\frac{d}{\phi q}-(1-\frac{1}{p})(1+\frac{d}{\phi q})\frac{d}{\theta}\text{ holds, if and only if }\theta+\frac{d}{p}<d+\phi q.
	\end{aligned}
\end{equation*}
Notice that $\frac{d}{q}+\phi=\frac{d}{p}+\theta$, so $\theta+\frac{d}{p}<d+\phi q$ is equivalent to $q>1$, which is one of the hypothesis. Also,
\begin{equation*}
	\begin{aligned}
	1+\frac{d}{\phi q}-(\frac{1}{2}-\frac{1}{p})(1+\frac{d}{\phi q})\frac{d}{\theta}>0\text{ holds, if and only if }\frac{\theta}{d}>\frac{1}{2}-\frac{1}{p},
	\end{aligned}
\end{equation*} 
which is our hypothesis. Other parts are almost trivial.
\item When $1<p<2$, it suffices to prove that there exists a $\delta>0$ so that
	\begin{equation*}
	\begin{aligned}
		1+\frac{d}{\phi q}-(1-\frac{1}{p})(1+\frac{d}{\phi q})\frac{d}{\theta}<\delta<\min\Big\{1+\frac{d}{\theta p},	1+\frac{d}{\phi q}\Big\}.
	\end{aligned}
\end{equation*}
We omit the details and point out that we only need to show
\begin{equation*}
		\min\Big\{1+\frac{d}{\theta p},	1+\frac{d}{\phi q}\Big\}>\max\Big\{0,	1+\frac{d}{\phi q}-(1-\frac{1}{p})(1+\frac{d}{\phi q})\frac{d}{\theta}\Big\}.
\end{equation*}
The above inequality is equivalent to $p,q>1$.
	\end{itemize}
\end{proof}
Once $\delta$ has been chosen, we define $\epsilon,\tilde{\epsilon}$ via
\begin{equation}\label{definition of epsilon and tilde(epsilon)}
	\frac{\epsilon}{d+\epsilon}=\delta\frac{\phi q}{d+\phi q},\text{ and }\frac{\tilde{\epsilon}}{d+\tilde{\epsilon}}=\delta\frac{\theta p}{d+\theta p}.
\end{equation}
Particularly, 
\begin{equation*}
	\frac{d}{\epsilon}=\frac{d+\phi q-\delta\phi q}{\delta\phi q}.
\end{equation*}
By Proposition \ref{choice of parameters: Cowling-Price}, the numerator $d+\phi q-\delta \phi q>0$, so $\epsilon>0$. By the same reason, we have $\tilde{\epsilon}>0$. 
\par Now we choose
\begin{equation*}
	\begin{aligned}
		a&:=\frac{p}{1+\frac{p\theta}{d+\epsilon}},\\ r&:=\frac{2}{a}\text{ and }s:=\frac{r}{r-1},\\
		T&:=c_d\bigg(\frac{||f||_a}{||f||_2}\bigg)^{\frac{as}{d}},
	\end{aligned}
\end{equation*}
where $c_d$ is the constant so that $(c_d^dv_d)^{1/s}=1/2$.
\par By H\"older inequality for the pair $(r,s)$ as before, we have
\begin{equation*}
	\frac{1}{2}||f||_a^a\leq\int_{|x|>T}|f|^a.
\end{equation*}
Now we choose $b>0$ and $r_1,s_1\in(1,\infty)$ so that
\begin{equation}\label{choose parameters: original side}
	\begin{aligned}
		ar_1&=p,\\
		s_1&=\frac{r_1}{r_1-1},\\
		br_1&=\theta p.\\	
	\end{aligned}
\end{equation}
Moreover, we have
\begin{equation*}
	bs_1=\frac{\theta p}{r_1-1}=\frac{\theta p}{\frac{p}{a}-1}=d+\epsilon.
\end{equation*}
Trivially, we have $b>0$, and we still need to verify $r_1>1$. This equivalent to $p>a$, which is true by the choice of $a$. Now, writing
\begin{equation*}
	\int_{|x|>T}|f|^a=\int_{|x|>T}\frac{1}{|x|^b}\cdot|x|^b|f|^a
\end{equation*}
and using H\"older inequality for the pair $(s_1,r_1)$, we can prove (recall $bs_1=d+\epsilon$ and $br_1=\theta p$)
\begin{equation*}
	\frac{1}{2}||f||_a^a\leq\bigg(\frac{\omega_{d-1}}{\epsilon T^\epsilon}\bigg)^\frac{1}{s_1}	\Big|\Big||x|^\theta f\Big|\Big|^a_p,
\end{equation*}
after a simplification
\begin{equation*}
 \frac{\Big|\Big||x|^\theta f\Big|\Big|_p}{||f||_2}\geq\bigg(\frac{\epsilon c_d^{\epsilon}}{2^{s_1}\omega_{d-1}}\bigg)^\frac{1}{as_1}\bigg(\frac{||f||_a}{||f||_2}\bigg)^{1+\frac{s\epsilon}{ds_1}}
\end{equation*}
\paragraph{}For the Fourier transform side, we choose
\begin{equation*}
	\begin{aligned}
		\tilde{a}:=\frac{q}{1+\frac{q\phi}{d+\tilde{\epsilon}}},\quad  \tilde{r}:=\frac{2}{\tilde{a}}\text{ and }\tilde{s}:=\frac{\tilde{r}}{\tilde{r}-1},\quad
		\tilde{T}:=\widetilde{c_d}\bigg(\frac{||\hat{f}||_{\tilde{a}}}{||\hat{f}||_2}\bigg)^{\frac{\tilde{a}\tilde{s}}{d}},
	\end{aligned}
\end{equation*}
Now we verify that $a=\tilde{a}$ under our choice of $\epsilon,\tilde{\epsilon}$. After taking the reciprocal, it suffices to prove
\begin{equation*}
	\frac{1}{q}+\frac{\phi}{d+\tilde{\epsilon}}=\frac{1}{p}+\frac{\theta}{d+\epsilon}.
\end{equation*}
Using the hypothesis $1/q+\phi/d=1/p+\theta/d$, we only need to prove $\frac{\phi}{d+\tilde{\epsilon}}-\frac{\phi}{d}=	\frac{\theta}{d+\epsilon}-\frac{\theta}{d}$, this is equivalent to
\begin{equation*}
	\frac{\phi\tilde{\epsilon}}{d+\tilde{\epsilon}}=\frac{\theta\epsilon}{d+\epsilon}.
\end{equation*}
The above equality is easy to prove by combining the hypothesis $1/q+\phi/d=1/p+\theta/d$ with the choice of $\epsilon$ and $\tilde{\epsilon}$, see (\ref{definition of epsilon and tilde(epsilon)}).
\par As an immediate corollary of $a=\tilde{a}$, we have
\begin{equation*}
	\tilde{r}=r,\quad \tilde{s}=s.
\end{equation*}
The constant $\widetilde{c_d}$ satisfies $(\widetilde{c_d}^dv_d)^{1/\tilde{s}}=1/2$, so $\widetilde{c_d}=c_d$.
Now parallel to previous proof, we obtain
\begin{equation*}
		\frac{1}{2}||\hat{f}||_{\tilde{a}}^{\tilde{a}}=\frac{1}{2}||\hat{f}||_a^a\leq\int_{|\xi|>\widetilde{T}}|\hat{f}|^a.
\end{equation*}
Now we choose $\tilde{b}>0$ and $\widetilde{r_1},\widetilde{s_1}\in(1,\infty)$ so that
\begin{equation}\label{choose parameters: Fourier transform side}
	\begin{aligned}
		\tilde{a}\widetilde{r_1}&=q,\\
		\widetilde{s_1}&=\frac{\widetilde{r_1}}{\widetilde{r_1}-1},\\
		\tilde{b}\widetilde{r_1}&=\phi q.\\	
	\end{aligned}
\end{equation}
Similar to $bs_1=d+\epsilon$, we have $\tilde{b}\widetilde{s_1}=d+\tilde{\epsilon}$. Now we can prove, by the same method,
\begin{equation*}
	\frac{1}{2}||\hat{f}||_a^a\leq\bigg(\frac{\omega_{d-1}}{\tilde{\epsilon}\big(\widetilde{T}\big)^{\tilde{\epsilon}}}\bigg)^{1/\widetilde{s_1}}	\Big|\Big||\xi|^\phi \hat{f}\Big|\Big|^a_q,
\end{equation*}
A simplification shows that
\begin{equation*}
	\frac{\Big|\Big||\xi|^\phi \hat{f}\Big|\Big|_q}{||\hat{f}||_2}\geq\bigg(\frac{\tilde{\epsilon}c_d^{\tilde{\epsilon}}}{2^{\widetilde{s_1}}\omega_{d-1}}\bigg)^{1/a\widetilde{s_1}}\bigg(\frac{||\hat{f}||_a}{||\hat{f}||_2}\bigg)^{s\tilde{\epsilon}/d\widetilde{s_1}+1}.
\end{equation*}
Take a product we get
\begin{equation*}
 \frac{\Big|\Big||x|^\theta f\Big|\Big|_p}{||f||_2}\frac{\Big|\Big||\xi|^\phi \hat{f}\Big|\Big|_q}{||\hat{f}||_2}\geq\bigg(\frac{\epsilon c_d^{\epsilon}}{2^{s_1}\omega_{d-1}}\bigg)^\frac{1}{as_1}\bigg(\frac{\tilde{\epsilon}c_d^{\tilde{\epsilon}}}{2^{\widetilde{s_1}}\omega_{d-1}}\bigg)^{1/a\widetilde{s_1}}\bigg(\frac{||f||_a}{||f||_2}\bigg)^{1+\frac{\epsilon s}{ds_1}}\bigg(\frac{||\hat{f}||_a}{||\hat{f}||_2}\bigg)^{\tilde{\epsilon} s/d\widetilde{s_1}+1}.
\end{equation*}
Notice that by our choice of $\epsilon,\tilde{\epsilon}$,
\begin{equation}\label{epsilon/s_1 is equal to tilde(epsilon)/tilde(s_1)}
	1+\frac{\epsilon s}{ds_1}=1+\frac{\tilde{\epsilon}s}{d\widetilde{s_1}}.
\end{equation}
To prove this, it suffices to check that $\epsilon/s_1=\tilde{\epsilon}/\widetilde{s_1}$. We recall $bs_1=d+\epsilon$ and $\tilde{b}\widetilde{s_1}=d+\tilde{\epsilon}$, so (\ref{epsilon/s_1 is equal to tilde(epsilon)/tilde(s_1)}) is equivalent to $\epsilon b/(d+\epsilon)=\tilde{\epsilon}\tilde{b}/(d+\widetilde{\epsilon})$. Also by (\ref{choose parameters: original side}) and (\ref{choose parameters: Fourier transform side}), we have $b=a\theta$ and $\tilde{b}=\tilde{a}\phi$. Finally, we already proved $a=\tilde{a}$, so (\ref{epsilon/s_1 is equal to tilde(epsilon)/tilde(s_1)}) is equivalent to
\begin{equation*}
	\frac{\epsilon \theta}{d+\epsilon}=\frac{\tilde{\epsilon}\phi}{d+\widetilde{\epsilon}}.
\end{equation*}
This equality is easy to verify by the definition of $\epsilon$ and $\tilde{\epsilon}$, see (\ref{definition of epsilon and tilde(epsilon)}).
\par So, by (\ref{epsilon/s_1 is equal to tilde(epsilon)/tilde(s_1)}), we get
\begin{equation*}
	\begin{aligned}
	 \frac{\Big|\Big||x|^\theta f\Big|\Big|_p}{||f||_2}\frac{\Big|\Big||\xi|^\phi \hat{f}\Big|\Big|_q}{||\hat{f}||_2}&\geq\bigg(\frac{\epsilon c_d^{\epsilon}}{2^{s_1}\omega_{d-1}}\bigg)^\frac{1}{as_1}\bigg(\frac{\tilde{\epsilon}c_d^{\tilde{\epsilon}}}{2^{\widetilde{s_1}}\omega_{d-1}}\bigg)^{1/a\widetilde{s_1}}\bigg(\frac{||f||_a||\hat{f}||_a}{||f||_2||\hat{f}||_2}\bigg)^{1+\frac{\epsilon s}{ds_1}}\\
	 &\geq\bigg(\frac{\epsilon c_d^{\epsilon}}{2^{s_1}\omega_{d-1}}\bigg)^\frac{1}{as_1}\bigg(\frac{\tilde{\epsilon}c_d^{\tilde{\epsilon}}}{2^{\widetilde{s_1}}\omega_{d-1}}\bigg)^{1/a\widetilde{s_1}},
 \end{aligned}
\end{equation*}
where the last inequality we use (\ref{L^2 estimate for f and its fourier transform}).

\subsection{Proof of Theorem \ref{Cowling-Price theorem, d-dimension}: Only if part}\label{Proof of Cowling-Price theorem:  Only if part}
\paragraph{}Now we consider the converse part. It suffices to prove that the theorem is not true if
\begin{equation*}
		\frac{\theta}{d}\leq\frac{1}{2}-\frac{1}{p}\quad\text{or}\quad\frac{\phi}{d}\leq\frac{1}{2}-\frac{1}{q}.
\end{equation*}
There are two cases. The first case is $	\frac{\theta}{d}=\frac{1}{2}-\frac{1}{p}$, and $\frac{\phi}{d}=\frac{1}{2}-\frac{1}{q}.$ The second case is $	\frac{\theta}{d}<\frac{1}{2}-\frac{1}{p}$, and $\frac{\phi}{d}<\frac{1}{2}-\frac{1}{q}.$ Notice that due to the homogeneous condition,
\begin{equation*}
	\frac{1}{q}+\frac{\phi}{d}=\frac{1}{p}+\frac{\theta}{d},
\end{equation*} 
other cases is impossible, for example, we do not need to consider  $\frac{\theta}{d}=\frac{1}{2}-\frac{1}{p}$, and $\frac{\phi}{d}<\frac{1}{2}-\frac{1}{q}$.

\subsubsection{The case $	\frac{\theta}{d}=\frac{1}{2}-\frac{1}{p}$, and $\frac{\phi}{d}=\frac{1}{2}-\frac{1}{q}.$}
\paragraph{}Our goal is to construct an $f$ so that
\begin{equation*}
	|x|^\theta f\in L^p(\mathbb{R}^d)\text{ and }|\xi|^\phi\hat{f}\in L^q(\mathbb{R}^d),\text{ but }f\notin L^2(\mathbb{R}^d).
\end{equation*}
For this, we need a theorem of Wainger. The following version is a special case of Wainger's theorem, where we choose $l=0$, and $b(t)$ in slowly varying class(\cite{Wainger}, Definition 2, page 7) is a logarithmic power.
\begin{thm}[\cite{Wainger}, Theorem 1, page 8]Let $b(t)=\log^{-\frac{1}{2}}(t+2)$. Let $\psi(t)$ be a smooth function with the following properties
	\begin{itemize}
		\item $\psi(t)=0$ when $t\leq\frac{1}{2}$.
		\item $\psi(t)=1$ when $t\geq1$.
		\item $0\leq\psi(t)\leq1$ for all $t\in\mathbb{R}$.
	\end{itemize} Assume that $-\infty<\gamma<d$, and $\gamma\neq0,-2,-4,...$ Define
	\begin{equation*}
		F_\epsilon(x):=\mathcal{F}^{-1}\Big(\frac{\psi(|\xi|)}{|\xi|^\gamma}b(|\xi|)\mathrm{e}^{-\epsilon|\xi|}\Big),
	\end{equation*}
where $\mathcal{F}^{-1}$ is the Fourier inversion. Then
\begin{itemize}
	\item[(1)] $F(x):=\lim_{\epsilon\rightarrow0^+}F_\epsilon(x)$ exists for all $x\neq0$, and $F(x)$ is $C^\infty$-differentiable when $x\neq0$.
	\item[(2)]For any integer $r$, we have $|F_\epsilon(x)|=\mathcal{O}(|x|^{-r})$ uniformly in $\epsilon\geq0$, as $|x|\rightarrow\infty$.
\end{itemize}
When $x$ is near $0$, we have
\begin{itemize}
	\item[(3)]$F(x)=C(\gamma,d)\cdot b\big(\frac{1}{|x|}\big)|x|^{\gamma-d}+E(x)$, where $C(\gamma,d)$ is a non-zero constant and $E(x)$ is the error term.
	\item[(4)]The error term satisfies $|E(x)|=\mathcal{O}(1)+o\Big(b\big(1/|x|\big)|x|^{\gamma-d}\Big)$ as $x\rightarrow0$.
\end{itemize}
Also, if $F(x)\in L^1(\mathbb{R}^d)$, then $F_\epsilon(x)$ is dominated by an integrable function, and 
\begin{equation*}
	\widehat{F}(\xi)=\frac{\psi(|\xi|)}{|\xi|^\gamma}b(|\xi|).
\end{equation*} 
\end{thm}
To apply this theorem, we define our $f$ via
\begin{equation*}
	f(x):=F(x)=\lim_{\epsilon\rightarrow0^+}F_\epsilon(x),\text{ with }\gamma=\frac{d}{2}.
\end{equation*}
However, the function $f$ may not in Schwartz class. This is a minor problem since we can approximate $f$ by elements in $\mathcal{S}(\mathbb{R}^d)$ via a density argument.
\paragraph{The function $f$ is not $L^2$ integrable:} When $x$ is near $0$, function $f$ behaves like
\begin{equation*}
	\frac{1}{|x|^{\frac{d}{2}}\log^\frac{1}{2}(\frac{1}{|x|}+2)}\sim\frac{1}{|x|^{\frac{d}{2}}\log^\frac{1}{2}(\frac{1}{|x|})},\text{ as }x\rightarrow0.
\end{equation*}
An easy calculation yields that the function
\begin{equation*}
	x\mapsto\frac{1}{|x|^{\frac{d}{2}}\log^\frac{1}{2}(|x|^{-1})}
\end{equation*}
is not $L^2$ integrable near $0$. So $f\notin L^2(\mathbb{R}^d)$.
\paragraph{The function $|x|^\theta|f|$ is $L^p$ integrable:}We notice that singularity appears only when $x\rightarrow0$ and $|x|\rightarrow\infty$, so we only need to prove the $L^p$-integrability of $|x|^\theta f$ when $x$ is near $0$, and when $|x|$ is large. When $x$ is near $0$, similar to the proof of $f\notin L^2$, the function $|x|^\theta|f|$ behaves like
\begin{equation*}
	\frac{1}{|x|^{\frac{d}{2}-\theta}\log^\frac{1}{2}(\frac{1}{|x|})},\text{ as }x\rightarrow0.
\end{equation*}
Taking $L^p$-norm, and using the hypothesis $\frac{\theta}{d}=\frac{1}{2}-\frac{1}{p}$, we obtain an integral
\begin{equation}\label{proof of L^A integrability at zero}
	\int_{x\in\mathbb{R}^d:|x|\ll1}\frac{\mathrm{d}x}{|x|^d\log^\frac{p}{2}(\frac{1}{|x|})}
\end{equation}
Here $|x|\ll1$ means that $|x|$ is much less than $1$. Notice that $\frac{\theta}{d}=\frac{1}{2}-\frac{1}{p}$ and $\theta>0$, we must have $p>2$. Therefore, (\ref{proof of L^A integrability at zero}) is finite, and $|x|^\theta f$ is $L^p$ integrable near $x=0$.
\par The $L^p$ integrability of $|x|^\theta f$ when $|x|\rightarrow\infty$ is easier. By Wainger's theorem, for every integer $r$, as $|x|\rightarrow\infty$,
\begin{equation*}
	|F_\epsilon(x)|=\mathcal{O}(|x|^{-r}), \text{ uniformly in } \epsilon\geq0.
\end{equation*}
So $f(x)=\mathcal{O}(|x|^{-r})$ also holds. Just take $r$ large, we see that $|x|^\theta|f|$ is $L^p$ integrable when $|x|\gg1$. Here $|x|\gg1$ means that $|x|$ is much larger than $1$.
\paragraph{}Finally, we prove $|\xi|^\phi\hat{f}$ is $L^q$ integrable. If we can prove $f\in L^1$, then by Wainger's theorem
\begin{equation*}
 \hat{f}(\xi)=\frac{\psi(|\xi|)}{|\xi|^\frac{d}{2}\log^\frac{1}{2}(|\xi|+2)}.
\end{equation*}
The function $|\xi|^\phi\hat{f}\in L^q$ is similar to the proof of $|x|^\theta f\in L^p$. Notice that $|\xi|^\phi\hat{f}$ has no singularity at $x=0$, since $\psi(t)=0$ when $t\leq1/2$. So it suffices to prove the $L^1$-integrability of $f$. When $x$ is near $0$, function $f$ behaves like
\begin{equation*}
	\frac{1}{|x|^\frac{d}{2}\log^\frac{1}{2}\Big(\frac{1}{|x|}\Big)},
\end{equation*}
and it's easy to show that
\begin{equation*}
	\int_{x\in\mathbb{R}^d:|x|\ll1}	\frac{\mathrm{d}x}{|x|^\frac{d}{2}\log^\frac{1}{2}\Big(\frac{1}{|x|}\Big)}<\infty.
\end{equation*}
When $|x|\gg1$, the $L^1$-integrability of $f$ is easier, just notice that $f(x)=\mathcal{O}(|x|^{-r})$ for every $r\in\mathbb{N}$.

\subsubsection{The case  $	\frac{\theta}{d}<\frac{1}{2}-\frac{1}{p}$, and $\frac{\phi}{d}<\frac{1}{2}-\frac{1}{q}.$}
We're going to use the Rudin-Shapiro construction, which is similar to the 1-dimensional case in \cite{Cowling-Price} (see \cite{Cowling-Price}, counterexample II). To avoid complicated notations, we only focus on the case $d=2$. After proving the two-dimensional case, we will give a brief explanation on general $d$.
\par Choose a non-negative smooth function $f$ which is supported in $[1/10,9/10]^2$ and $||f||_\infty=1$, we define $\{f_k\}^\infty_{k=0},\{g_k\}^\infty_{k=0}$, $\{p_k\}^\infty_{k=0},\{q_k\}^\infty_{k=0}$ inductively:
\begin{itemize}
	\item When $k=0$, define $f_0=g_0=p_0=q_0:=f$.
	\item Assume that the four sequences have been defined for $0,1,...,k$, now we define
	\begin{equation*}
		\begin{aligned}
			f_{k+1}&:=f_k+\Big[g_k(x-(2^k,0))+p_k(x-(0,2^k))+q_k(x-(2^k,2^k))\Big]\\
			g_{k+1}&:=f_k+\Big[g_k(x-(2^k,0))-p_k(x-(0,2^k))-q_k(x-(2^k,2^k))\Big]\\
			p_{k+1}&:=f_k-\Big[g_k(x-(2^k,0))-p_k(x-(0,2^k))+q_k(x-(2^k,2^k))\Big]\\
			q_{k+1}&:=f_k-\Big[g_k(x-(2^k,0))+p_k(x-(0,2^k))-q_k(x-(2^k,2^k))\Big]\\
		\end{aligned}
	\end{equation*}
\end{itemize}
Notice that
\begin{equation*}
	\begin{aligned}
		\widehat{f_{k+1}}&=\widehat{f_k}+\mathrm{e}^{-2\pi\mathrm{i}2^k\xi_1}\widehat{g_k}+\mathrm{e}^{-2\pi\mathrm{i}2^k\xi_2}\widehat{p_k}+\mathrm{e}^{-2\pi\mathrm{i}2^k(\xi_1+\xi_2)}\widehat{q_k}\\
		\widehat{g_{k+1}}&=\widehat{f_k}+\mathrm{e}^{-2\pi\mathrm{i}2^k\xi_1}\widehat{g_k}-\mathrm{e}^{-2\pi\mathrm{i}2^k\xi_2}\widehat{p_k}-\mathrm{e}^{-2\pi\mathrm{i}2^k(\xi_1+\xi_2)}\widehat{q_k}\\
		\widehat{p_{k+1}}&=\widehat{f_k}-\mathrm{e}^{-2\pi\mathrm{i}2^k\xi_1}\widehat{g_k}+\mathrm{e}^{-2\pi\mathrm{i}2^k\xi_2}\widehat{p_k}-\mathrm{e}^{-2\pi\mathrm{i}2^k(\xi_1+\xi_2)}\widehat{q_k}\\
		\widehat{q_{k+1}}&=\widehat{f_k}-\mathrm{e}^{-2\pi\mathrm{i}2^k\xi_1}\widehat{g_k}-\mathrm{e}^{-2\pi\mathrm{i}2^k\xi_2}\widehat{p_k}+\mathrm{e}^{-2\pi\mathrm{i}2^k(\xi_1+\xi_2)}\widehat{q_k}\\
	\end{aligned}
\end{equation*}
By parallelogram law, for $a,b,c,d\in\mathbb{C}$
\begin{equation}\label{parallelogram law for 4 numbers}
	|a+b+c+d|^2+|a+b-c-d|^2+|a-b+c-d|^2+|a-b-c+d|^2=4(|a|^2+|b|^2+|c|^2+|d|^2).
\end{equation}
Let $(a,b,c,d)=\Big(\widehat{f_k},\ \mathrm{e}^{-2\pi\mathrm{i}2^k\xi_1}\widehat{g_k},\ \mathrm{e}^{-2\pi\mathrm{i}2^k\xi_2}\widehat{p_k},\ \mathrm{e}^{-2\pi\mathrm{i}2^k(\xi_1+\xi_2)}\widehat{q_k}\Big)$, we get
\begin{equation}\label{Rudin-Shapiro fourier side}
	|\widehat{f_{k+1}}|^2+	|\widehat{g_{k+1}}|^2+	|\widehat{p_{k+1}}|^2+	|\widehat{q_{k+1}}|^2=4\Big(	|\widehat{f_{k}}|^2+	|\widehat{g_{k}}|^2+	|\widehat{p_{k}}|^2+	|\widehat{q_{k}}|^2\Big).
\end{equation}
Furthermore, the function $f_k$ is supported in $[0,2^k]^2$. When $0\leq m,n\leq2^k-1$, 
\begin{equation*}
	f_k\Big|_{[m,m+1]\times[n,n+1]}=\pm f(x-(m,n)),
\end{equation*}
that is to say, in every cubes $Q_{m,n}:=[m,m+1]\times[n,n+1]\subset[0,2^k]^2$, function $f_k$ is obtained, up to a sign, via a translation of $f$ along the vector $(m,n)$. Therefore,
\begin{equation}\label{L^2 norm of function f_k}
	||f_k||_2^2=2^{2k}||f||_2^2,\text{ and }||f_k||_\infty=||f||_\infty=1.
\end{equation}
\paragraph{}Now we calculate the $L^p$ norm of $|x|^\theta f_k$:
\begin{equation*}
	\begin{aligned}
		\bigg(\int_{\mathbb{R}^d}|x|^{\theta p}|f_k|^p\mathrm{d}x \bigg)^\frac{1}{p}=&\bigg(\int_{[0,2^k]^2}|x|^{\theta p}|f_k|^p\mathrm{d}x \bigg)^\frac{1}{p}\\
		\leq&||f_k||_\infty\bigg(\int_{[0,2^k]^2}|x|^{\theta p}\mathrm{d}x \bigg)^\frac{1}{p}\\
		\leq&\bigg(\int_{|x|\leq\sqrt{d}2^k}|x|^{\theta p}\mathrm{d}x \bigg)^\frac{1}{p}\\
		\leq& C(d,\theta,p)\cdot 2^{k(\frac{d}{p}+\theta)}.
	\end{aligned}
\end{equation*}
Here $C(d,\theta,p)$ is a constant independent of $k$, here $d=2$.
\paragraph{}By (\ref{Rudin-Shapiro fourier side}), and recall $f_0=g_0=p_0=q_0=f$, we have
\begin{equation}\label{Control of fourier transform of f_k}
	|\widehat{f_{k}}|^2\leq	|\widehat{f_{k}}|^2+	|\widehat{g_{k}}|^2+|\widehat{p_{k}}|^2+|\widehat{q_{k}}|^2=4^{k+1}|\hat{f}|^2,
\end{equation}
which implies that $|\widehat{f_k}|\leq 2^{k+1}|\hat{f}|$. The $L^q$ norm of $|\xi|^\phi\hat{f}$ can be controlled:
\begin{equation*}
	\bigg(\int_{\mathbb{R}^d}|\xi|^{\phi q}|\widehat{f_{k}}|^q\mathrm{d}\xi\bigg)^\frac{1}{q}\leq2^{k+1}\bigg(\int_{\mathbb{R}^d}|\xi|^{\phi q}|\widehat{f}|^q\mathrm{d}\xi\bigg)^\frac{1}{q}=2^{k+1}\Big|\Big||\xi|^\phi\hat{f}\Big|\Big|_q.
\end{equation*}
\paragraph{}Now we test (\ref{Cowling-Price inequalitty, d-dimension}) by $f_k$, recall that $||f_k||_2^2=2^{2k}||f||^2_2$.
\begin{equation}\label{Cowling-Price counterexample for f_k}
	\begin{aligned}
	C\cdot2^{2k}||f||^2_2\leq\Big|\Big||x|^\theta f_k\Big|\Big|_p\cdot\Big|\Big||\xi|^\phi\widehat{f_k}\Big|\Big|_q\leq2^{k+1}\Big|\Big||\xi|^\phi\hat{f}\Big|\Big|_q\cdot  C(d,\theta,p)\cdot 2^{k(\frac{d}{p}+\theta)}.
\end{aligned}
\end{equation}
After a simplification, inequality (\ref{Cowling-Price counterexample for f_k}) is reduced to
\begin{equation}\label{Cowling-Price counterexample for f_k:reduced form}
	2^k\leq C_{d,\theta,p,\phi,q,f}\cdot2^{k(\frac{d}{p}+\theta)}.
\end{equation}
The constant $C_{d,\theta,p,\phi,q,f}$ is independent of $k$, and the above inequality should hold for all $k$. This is impossible, because when $d=2$,
\begin{equation*}
\text{the hypothesis}\quad\frac{\theta}{d}<\frac{1}{2}-\frac{1}{p}\quad\text{implies}\quad\frac{d}{p}+\theta<1.
\end{equation*}
Let $k$ goes to infinity will give a contradiction. Therefore, the inequality (\ref{Cowling-Price inequalitty, d-dimension}) can not hold for $f_k$ when $k$ is large.
\paragraph{}For general $d\geq1$, we need to construct $2^d$ families of functions
\begin{equation*}
	\{f_{1,k}\}_{k=0}^\infty,\ \{f_{2,k}\}_{k=0}^\infty,...,\ \{f_{2^d,k}\}_{k=0}^\infty.
\end{equation*}
One may think, when $d=2$,
\begin{equation*}
	\begin{aligned}
			\{f_{1,k}\}_{k=0}^\infty&=\{f_k\}_{k=0}^\infty,\  \{f_{2,k}\}_{k=0}^\infty=\{g_k\}_{k=0}^\infty,\\ 	\{f_{3,k}\}_{k=0}^\infty&=\{p_k\}_{k=0}^\infty,\  \{f_{4,k}\}_{k=0}^\infty=\{q_k\}_{k=0}^\infty.\
	\end{aligned}
\end{equation*}
The reason for using $2^d$ families is because a cube in $\mathbb{R}^d$, for example, $[0,2)^d$ has $2^d$ dyadic children. Each of its children has the form
\begin{equation*}
	I_1\times I_2\times\cdots\times I_d,\text{ where }I_i=[0,1)\text{ or }[1,2).
\end{equation*}
Another ingredient is a parallelogram law for $2^d$ numbers.
\begin{prop}\label{Proposition: parallelogram law for 2^d numbers}
For every $d\geq1$, there exists a choice of the signs $\{\epsilon_{i,j}\}_{1\leq i\leq2^d,\\2\leq j\leq2^d}$(notice that $i$ starts from $1$, while $j$ starts from $2$), with $\epsilon_{i,j}=\pm1$ for all $i,j$, so that for every $a_1,a_2,...,a_{2^d}\in\mathbb{C}$, we have
	\begin{equation}\label{parallelogram law for 2^d numbers}
		\Big|a_1+\epsilon_{1,2}a_2+\cdots+ \epsilon_{1,2^d}a_{2^d}\Big|^2+\cdots+\Big|a_1+\epsilon_{2^d,2}a_2+ \cdots+ \epsilon_{2^d,2^d}a_{2^d}\Big|^2=2^d(|a_1|^2+|a_2|^2+\cdots|a_{2^d}|^2).
	\end{equation}
\end{prop}
\begin{proof}
	We prove it by induction, when $d=1$, this is just the ordinary parallelogram law 
	\begin{equation*}
		|a_1+a_2|^2+|a_1-a_2|^2=2(|a_1|^2+|a_2|^2).
	\end{equation*}
Also, we already saw the case $d=2$ in (\ref{parallelogram law for 4 numbers}). Assume that the proposition is true for $1,2,...,d$, now we prove the case $d+1$. The left hand side of (\ref{parallelogram law for 2^d numbers}) has exactly $2^d$ terms, and we may write it as
\begin{equation}\label{parallelogram law for 2^d numbers, a form of general term}
	\sum_{i=1}^{2^d}\Big|a_1+\epsilon_{i,2}a_2+\cdots+ \epsilon_{i,2^d}a_{2^d}\Big|^2=\sum_{i=1}^{2^d}\Big|a_1+\sum_{j=2}^{2^d}\epsilon_{i,j}a_j\Big|^2.
\end{equation}
For each term $|a_1+\epsilon_{i,2}a_2+\cdots+ \epsilon_{i,2^d}a_{2^d}|^2$ in (\ref{parallelogram law for 2^d numbers, a form of general term}), we define \textbf{two} new terms
\begin{equation}\label{parallelogram law for 2^{d+1} numbers, a form of general term}
	\begin{aligned}
	&\Big|(a_1+\epsilon_{i,2}a_2+\cdots+ \epsilon_{i,2^d}a_{2^d})+(a_{2^d+1}+\epsilon_{i,2}a_{2^d+2}\cdots+ \epsilon_{i,2^d}a_{2^{d+1}})\Big|^2+\\	&\Big|(a_1+\epsilon_{i,2}a_2+\cdots+ \epsilon_{i,2^d}a_{2^d})-(a_{2^d+1}+\epsilon_{i,2}a_{2^d+2}\cdots+ \epsilon_{i,2^d}a_{2^{d+1}})\Big|^2.
\end{aligned}
\end{equation}
This is exactly our choice of the signs in $d+1$ case and finishes the proof. In fact, by parallelogram law,
\begin{equation*}
	(\ref{parallelogram law for 2^{d+1} numbers, a form of general term})=2\Big(|a_1+\epsilon_{i,2}a_2+\cdots+ \epsilon_{i,2^d}a_{2^d}|^2+|a_{2^d+1}+\epsilon_{i,2}a_{2^d+2}\cdots+ \epsilon_{i,2^d}a_{2^{d+1}}|^2\Big).
\end{equation*}
 So after taking a sum, we get
\begin{equation*}
	\begin{aligned}
	\sum^{2^d}_{i=1}(\ref{parallelogram law for 2^{d+1} numbers, a form of general term})=&2\times	\sum^{2^d}_{i=1}	|a_1+\epsilon_{i,2}a_2+\cdots+ \epsilon_{i,2^d}a_{2^d}|^2+2\times	\sum^{2^d}_{i=1}	|a_{2^d+1}+\epsilon_{i,2}a_{2^d+2}\cdots+ \epsilon_{i,2^d}a_{2^{d+1}}|^2\\
	=&2\cdot2^d(|a_1|^2+\cdots+|a_{2^d}|^2)+2\cdot2^d(|a_{2^d+1}|^2+\cdots+|a_{2^{d+1}}|^2)\\
	=&2^{d+1}(|a_1|^2+\cdots+|a_{2^{d+1}}|^2).
	\end{aligned}
\end{equation*}
Here, in the second-to-last equality, we used the induction hypothesis for $d$.
\end{proof}

\par Choose a smooth function $f$ supported in $[1/10,9/10]^d$ and $||f||_\infty=1$. We define, inductively,
\begin{itemize}
	\item When $k=0$, the functions $f_{1,0}=f_{2,0}=...=f_{2^d,0}:=f$, 
	\item Assume that $2^d$ sequences have been defined for $0,1,...,k$, now we define
		\begin{equation*}
		\begin{aligned}
			f_{i,k+1}(x)&:=f_{1,k}(x)+\sum_{j=2}^{2^d}\epsilon_{i,j}f_{j,k}(x-\mathbf{x_{j,k}}),\text{ where }1\leq i\leq 2^d.
		\end{aligned}
	\end{equation*}
As $j$ runs over $2$ to $2^d$, the point $\mathbf{x_{j,k}}\in\mathbb{R}^d$ runs over all $2^d-1$ integer points of the form
\begin{equation*}
	\mathbf{x_{j,k}}=(x_1,x_2,...,x_d),\text{ each }x_i=0\text{ or }2^k, \text{ and at least one }x_i\neq0.
\end{equation*}
\end{itemize} 
It's easy to see when $d=2$, our $\mathbf{x_{j,k}}$ are exactly the three points $(2^k,2^k)$, $(0,2^k)$, and $(2^k,0)$. The subsequent calculation is the same as for $d=2$, so we omit it here. Just replace (\ref{L^2 norm of function f_k}) and (\ref{Control of fourier transform of f_k}) by
\begin{equation*}
	||f_{1,k}||^2_2=2^{dk}||f||_2^2,\text{ and }|\widehat{f_{1,k}}|^2\leq 2^{d(k+1)}|\hat{f}|^2
\end{equation*}
respectively.

\end{document}